\documentclass[12pt]{amsart}

\usepackage{fullpage}
\usepackage[leqno]{amsmath}
\usepackage{amsfonts}
\usepackage{amssymb}
\usepackage{amsthm}
\usepackage{amssymb}
\usepackage{bbm}

\theoremstyle{plain}
\newtheorem{theorem}{Theorem}[section]
\newtheorem{corollary}[theorem]{Corollary}
\newtheorem{lemma}[theorem]{Lemma}

\theoremstyle{definition}

\newtheorem{example}[theorem]{Example}

\DeclareMathOperator{\spn}{span}

\title{Nilpotent Lie and Leibniz algebras}

\author[Batten Ray]{Chelsie Batten Ray}
\thanks{The work of the first three authors was supported by NSF grant DMS-0943855.}
\address{Department of Mathematics, North Carolina State University\\
Raleigh, NC 27695}
\email{cabatten@ncsu.edu}

\author[Combs]{Alexander Combs}
\address{Department of Mathematics, North Carolina State University\\
Raleigh, NC 27695}
\email{ancombs2@ncsu.edu}

\author[Gin]{Nicole Gin}
\address{Department of Mathematics, North Carolina State University\\
Raleigh, NC 27695}
\email{nagin@ncsu.edu}

\author[Hedges]{Allison Hedges}
\address{Department of Mathematics, North Carolina State University\\
Raleigh, NC 27695}
\email{armcalis@ncsu.edu}

\author[Hird]{J.T. Hird}
\address{Department of Mathematics, West Virginia University, Institute of Technology\\
Montgomery, WV 25136}
\email{jthird@ncsu.edu}

\author[Zack]{Laurie Zack}
\address{Department of Mathematics and Computer Science, High Point University\\
High Point, NC 27262}
\email{lzack@highpoint.edu}

\begin{document}
\maketitle

\begin{abstract}
We extend results on finite dimensional nilpotent Lie algebras to Leibniz algebras and counterexamples to others are found.  One generator algebras are used in these examples and are investigated further.
\end{abstract}

\section{Introduction}\label{intro}


Results on nilpotent Lie algebras have been extended to Leibniz algebras by various authors.  Among them are Engel's Theorem (\cite{ayupov}, \cite{barnesengel}, \cite{jacobsonleib}, \cite{patso}), the nilpotency of algebras which admit a prime period automorphism without non-zero fixed points \cite{jacobsonleib}, and the equivalence of 
(a) nilpotency, (b) the normalizer condition, (c) the right normalizer condition, (d) that all maximal subalgebras are ideals, and (e) that all maximal subalgebras are right ideals (which follows from material in \cite{barnesleib}).
We will consider extending other results of this type.  In some cases, we have found such extensions.  In other cases, the results do not extend, one generator algebras providing examples.  Investigations of these one generator algebras occur in \cite{barneslattice} and \cite{allison}.  We add further results in this direction.  We also consider Leibniz algebras whose center or right center are one dimensional and extend non-embedding results from Lie theory (\cite{chao}, \cite{stitz1}).  There are results in \cite{ao} for when the left center is one dimensional. (In \cite{patso} the work is done for Leibniz algebras for which right multiplication is a derivation as opposed to, following Barnes, our algebra for which left multiplication is a derivation.  Thus in \cite{patso} our left center would be right center, etc.)  In this work we consider only finite dimensional Leibniz algebras.

\section{Preliminaries}

Let $\Phi(A)$ be the Frattini subalgebra of the Leibniz algebra $A$, which is the intersection of all maximal subalgebras of $A$.  As in Lie theory, $\Phi(A)$ is an ideal when the algebra is of characteristic 0 \cite{REG2011}, but not generally, even if the algebra is solvable \cite{REG2011}, which is counter to the case for solvable Lie algebras.

We will consider left Leibniz algebras, following Barnes \cite{barnesleib}. Hence a Leibniz algebra is an algebra that satisfies the identity $x(yz)=(xy)z +y(xz)$. We consider only finite dimensional algebras over a field. Let $A$ be a Leibniz algebra. The center of $A$ will be denoted by $Z(A)$ and $R(A)$ will be the right center,  $\{a \in A : Aa=0\}$.  We let $A^2=AA$ and define the lower central series by $A^{j+1}=AA^j$. It is known that $A^j$ is the space of all linear combination of products of $j$ elements no matter how associated. Thus it is often sufficient to consider only left-normed products of elements.  $A$ is nilpotent of class $t$ if $A^{t+1}=0$ but $A^t \neq 0$. When $A$ has class $t$, then $A^t \subset Z(A)$.

\section{Extending Lie Nilpotency Properties}

The first three results are direct generalizations of Lie algebra results; the proofs are similar to the Lie case, thus we omit them. The next two examples demonstrate results of Lie algebras that do not extend to Leibniz algebras. The remainder of the section is devoted to modifying the definition used in the last example in order that we can obtain an extension to Leibniz algebras, yet when restricted to Lie algebras, it is the original one. 

The first two results extend Lie algebra results from \cite{chaoS*}.

\begin{theorem}
Let $A$ be a Leibniz algebra and $N$ be a nilpotent ideal of $A$.  Then $A$ is nilpotent if and only if $A/N^2$ is nilpotent. 
\end{theorem}

\begin{corollary}
Let $A$ be a Leibniz algebra and $N$ be an ideal of $A$.  Suppose that $N$ is nilpotent of class $c$ and $A/N^2$ is nilpotent of class $d+1$.  Then $A$ is nilpotent of class at most ${c+1 \choose 2}d - {c \choose 2}$.
\end{corollary}

We say that a Leibniz algebra $A$ satisfies condition $k$ if the only subalgebra $K$ of $A$ with the property $K+A^2$ is $K=A$.  The proof of the next result is the same as in the Lie algebra case \cite{chaostitz}, and follows from the result of \cite{barnesleib} mentioned in the introduction.

\begin{theorem}
Let $A$ be a Leibniz algebra.  Then $A$ is nilpotent if and only if $A$ satisfies condition $k$.
\end{theorem}


Let $S$ be a subset of the Lie algebra $A$. The normal closure, $S^{A}$, of $S$ is the smallest ideal of $A$ that contains $S$. A Lie algebra, $A$, is nilpotent if and only if there is exactly one non-zero nilpotent subalgebra whose normal closure is $A$ \cite{chaostitz} (there is a requirement that the dimension of $A$ is large compared to the cardinality of the field). This result fails for Leibniz algebras.

\begin{example}
Let $A$ be a Leibniz algebra with basis $\{a,a^2\}$, and $aa^2=a^2$. $H= \spn\{a-a^2\}$ is a nilpotent subalgebra and, since $Ha$ is not contained in $H$, $H^A=A$. $H$ is unique with respect to this property. For any proper subalgebra is of the form $J=\spn\{a+\alpha a^2\}$, and $J$ is a subalgebra if and only if $\alpha=-1$. Therefore although $A$ is not nilpotent, the nilpotent subalgebra whose normal closure is $A$ is unique.
\end{example}

A Lie algebra is an $S^*$ algebra if each non-abelian subalgebra $H$ has $\dim(H/H^2) \geq 2$. A Lie algebra is an $S^*$ algebra if and only if it is nilpotent \cite{chaoS*} . This result does not extend directly to Leibniz algebras.

\begin{example}
Let $A$ be a Leibniz algebra with basis $\{a, a^2\}$, and $aa^2=0$.  Then $A$ is nilpotent, however $\dim(A/A^2)=1$.
\end{example}

We can alter the definition of an $S^*$ algebra to obtain a property equivalent to nilpotency 
which restricts to the original definition in the Lie algebra case.
Define a Leibniz algebra to be an $S^*$ algebra if every proper non-abelian subalgebra $H$ has either $\dim(H/H^2) \geq 2$ or $H$ is nilpotent and generated by one element.

\begin{theorem}\label{S*}
A Leibniz algebra is an S$^*$ algebra if and only if it is nilpotent.
\end{theorem}

\begin{lemma}\label{S*1}
Let $A$ be a non-abelian nilpotent Leibniz algebra. Then either $\dim(A/A^2) \geq 2$ or $A$ is generated by one element. 
\end{lemma}

\begin{proof}
Since $A$ is nilpotent, $A^2= \Phi(A)$, the Frattini subalgebra of $A$. Clearly $\dim(A/A^2) \neq 0$ since $A$ is nilpotent. If $\dim(A/A^2) = 1$, then $A$ is generated by one element. Otherwise $\dim(A/A^2) \geq 2$. 
\end{proof}

\begin{lemma}\label{S*2}
If $A$ is not nilpotent but all proper subalgebras of $A$ are nilpotent, then $\dim(A/A^2) \leq 1$.
\end{lemma}

\begin{proof}
Suppose that $\dim(A/A^2) \geq 2$. Then there exist distinct maximal subalgebras, $M$ and $N$ which contain $A^2$. Hence $M$ and $N$ are ideals and $A=M+N$ is nilpotent
, a contradiction.
\end{proof}

\begin{proof}[Proof of Theorem \ref{S*}] If $A$ is nilpotent, then every subalgebra is nilpotent, so $A$ is an $S^*$ algebra by Lemma \ref{S*1}. Conversely, suppose that there exists an $S^*$ algebra that is not nilpotent. Let $A$ be one on smallest dimension. All proper subalgebras of $A$ are $S^*$ algebras, hence are nilpotent. Thus $\dim(A/A^2) \leq 1$ by Lemma \ref{S*2}. Since $A$ is an $S^*$ algebra, it is generated by one element and is nilpotent, a contradiction.
\end{proof}

\section{Cyclic Leibniz algebra}

In the last section, we found that Leibniz algebras generated by one element provide counterexamples to the extension of several results from Lie to Leibniz algebras.  It would seem to be of interest to find properties of these algebras.  In this section we study them in their own right.

Let $A$ be a cyclic Leibniz algebra generated by $a$ and let $L_a$ denote left multiplication on $A$ by $a$.  Let $\{ a, a^2, \ldots, a^n \}$ be a basis for $A$ and $aa^n = \alpha_1 a + \cdots + \alpha_n a^n$.  The Leibniz identity on $a$, $a^2$, and $a$ shows that $\alpha_1 = 0$.  
Thus $A^2$ has basis $\{ a^2, \ldots, a^n \}$.  Let $T$ be the matrix for $L_a$ with respect to $\{ a, a^2, \ldots, a^n \}$.  $T$ is the companion matrix for $p(x) = x^n - \alpha_n x^{n-1} - \cdots - \alpha_2 x = p_1(x)^{n_1} \cdots p_s(x)^{n_s},$ where the $p_j$ are the distinct irreducible factors of $p(x)$.  We will continue using this notation throughout this section.  We will show:

\begin{theorem}
Let $A$ be a cyclic Leibniz algebra generated by $a$, and notation as in the last paragraph.  Then $\Phi(A) = \{ b \in A : q(L_a) (b) = 0 \},$ where $q(x) = p_1(x)^{n_1-1} \cdots p_s(x)^{n_s-1}$.
\end{theorem}

\begin{proof}
Let $A = W_1 \oplus \cdots \oplus W_s$ be the associated primary decomposition of $A$ with respect to $L_a$.  Then $W_j = \{ b \in A : p_j(L_a)^{n_j}(b) = 0 \}$.  Here $p(x)$ is also the minimal polynomial for $L_a$ on $A$, and therefore each $W_j$ is of the form $0 \subset U_{j,1} \subset \cdots \subset U_{j,n_j} = W_j$, where $U_{j,i} = \{ b \in A : p_j(L_a)^{i}(b) = 0 \}$, each $U_{j,i+1}/U_{j,i}$ is irreducible under the induced action of $L_a$, and $\dim(U_{j,i}) = i \deg(p_j(x))$.  Since $x$ is a factor of $p(x)$, we let $p_1(x) = x$.  
For $j \geq 2$, $W_j \subset A^2$ and for $i \neq n_1$, $U_{1,i} \subset A^2$.  
$A^2$ is abelian and left multiplication by $b \in A^2$ has $L_b = 0$ on $A$.  Hence
\begin{center}
\begin{tabular}{rcll}
$W_j  W_k$ &=& 0 & \text{ for } $1 \leq j,k \leq s$\\
$W_j \, W_1$ &=& 0 & \text{ for } $2 \leq j$\\
$W_1 \, W_j$ &$\subset$& $W_j$ & \text{ for } $1 \leq j \leq s$\\
$U_{1,n_1-1} \, W_j$ &=& 0 & \text{ for } $2 \leq j.$
\end{tabular}
\end{center}
Hence each $U_{j,i}$ except $U_{1,n_1} = W_1$ is an ideal in $A$.  $W_1$ is generally not a right ideal.

Let $M_j = W_1 \oplus \cdots \oplus U_{j,n_j-1} \oplus \cdots \oplus W_s$.  Since $\dim(A/A^2) = 1$, $A^2 = U_{1,n_1-1} \oplus W_2 \oplus \cdots \oplus W_s$, and $M_1$ is a maximal subalgebra of $A$.  We show that $M_j$, $j \geq 2$, is a maximal subalgebra of $A$.  Since $a=b+c$, where $b \in W_1$ and $c \in A^2$, $L_a = L_b$.  It follows that any subalgebra that contains $W_1$ is $L_a$ invariant.  If $M$ is a subalgebra of $A$ that contains $M_j$ properly, then $M \cap U_{j,n_j}$  contains $U_{j,n_j-1}$ properly. Since $U_{j,n_j}/U_{j,n_j-1}$ is irreducible in $A/U_{j,n_j-1}$, $M \cap U_{j,n_j} = U_{j,n_j}$ and $M=A$. Thus each $M_j$ is maximal in $A$ and $\Phi(A) \subset \cap M_j$.

Let $M$ be a maximal subalgebra of $A$. If $M=A^2$, then $M=\{b : g(L_a)(b)=0\}$ where $g(x)=p(x)/x$, so $M=M_1$. Suppose that $M \neq A^2$. Then $A=M+A^2$. Hence $a=m+c$, $m \in M$ and $c \in A^2$, and $L_a=L_m$. Hence $M$ is invariant under $L_a$. Thus the minimum polynomial $g(x)$ for $L_a$ on $M$ divides $p(x)$. If there is a polynomial $h(x)$ properly between $g(x)$ and $p(x)$, then the space $H$ annihilated by $h(L_a)$ is properly between $M$ and $A$, is invariant under left multiplications by $a$ and by any element in $A^2$, hence by any element in $A$ so it is a subalgebra. Since the minimum polynomial and characteristic polynomial for $L_a$ on $A$ are equal, the same is true of invariant subspaces.  Hence $\dim(H) = \deg(h(x))$ and $H$ is properly between $M$ and $A$, a contradiction.  Hence $M$ is the space annihilated by $g(x) = p(x) / p_j(x)$ for some $j$ and $M=M_j$. Therefore $\Phi(A)=\cap_{j=1}^s M_j = \{b \in A : q(L_a)(b) = 0\}$ where $q(x)=p_1(x)^{n_1-1} \cdots p_s(x)^{n_s-1}$.
%
%
\end{proof}


As a special case we obtain Corollary 3 of \cite{allison}.

\begin{corollary}
$\Phi(A)=0$ if and only if $p(x)$ is the product of distinct prime factors.
\end{corollary}

\begin{corollary}
The maximal subalgebras of $A$ are precisely the null spaces of $r_j(L_a)$, where $r_j(x) = p(x) / p_j(x)$ for $j = 1, \ldots, s$.
\end{corollary}

Now let $A_0$ and $A_1$ be the Fitting null and one components of $L_a$ acting on $A$. Since $L_a$ is a derivation of $A$, $A_0$ is a subalgebra of $A$. $L_a$ acts nilpotently on $A_0$ and $L_b =0$ when $b \in A^2$. Therefore for each $c \in A$, $L_c$ is nilpotent on $A_0$ and $A_0$ is nilpotent by Engel's theorem.  Let $a=b+c$, where $b \in A_0$ and $c \in A_1$. Then $L_a=L_b$ since $A_1 \subset A^2$ yields that $L_c=0$. Then $bA_1=aA_1=A_1$. For any non-zero $x \in A_1$, $bx$ is non-zero in $A_1$ and $x$ is not in the normalizer of $A_0$. Hence $A_1 \cap N_A(A_0) = 0$.
and $A_0 = N_A(A_0)$. Hence $A_0$ is a Cartan subalgebra of $A$.

Conversely let $C$ be a Cartan subalgebra of $A$ and $c \in C$. Then $c=d+e$, $d \in A_0$ and $e \in A_1$. Since $A_1 \subset A^2$, $A_1A_1=A_1A_0=0$, $A_0A_1=A_1$, and $A_0A_0 \subset A_0$. Therefore $A_1$ is an abelian subalgebra of $A$. Now $0= L_c^n(c)=L_d^n(c)=L_d^n(d+e)=L_d^n(e)=L_c^n(e)$, where we used that $eA=0$ and $d \in A_0$ which is nilpotent. Since $L_c$ is non-singular on $A_1$, $e=0$ and $c=d$. Hence $C \subset A_0$. Since $C$ is a Cartan subalgebra and $A_0$ is nilpotent, $C=A_0$ and $A_0$ is the unique Cartan subalgebra of $A$.

\begin{theorem} $A$ has a unique Cartan subalgebra. It is the Fitting null component of $L_a$ acting on $A$.
\end{theorem}

Using these same ideas, it can be shown that:

\begin{corollary}
The minimal ideals of $A$ are precisely $I_j=\{b \in A : p_j(L_a)(b)=0\}$ for $j >1$ and, if $n_1>1$, $I_1=\{b \in A : p_1(L_a)(b)=0\}.$ 
\end{corollary}

\begin{corollary}
$Asoc(A)=\{b \in A : u(L_a)(b)=0\}$, where $u(x)=p_2(x) \ldots p_s(x)$ if $n_1=1$ and $u(x)=p_1(x) \ldots p_s(x)$ otherwise.
\end{corollary}

\begin{corollary}
The unique maximal ideal of $A$ is $M_1 = \{b \in A : t(L_a)(b) =0\}$, where $t(x)=p(x)/p_1(x)$.
\end{corollary}

\section{Non-embedding}

Let $A$ be a Leibniz algebra.  Define the upper central series as usual; that is, let $Z_{1}(A) = \{z \in A : zA = Az = 0\}$ and inductively, $Z_{j+1}(A) = \{z \in A : Az \text{ and } zA \subset Z_{j}(A) \}$.  If $A$ is an ideal in a Leibniz algebra $N$, then the terms in the upper central series of $A$ are ideals in $N$.  Suppose that $A$ is nilpotent of dimension greater than one and $\dim(Z_{1}(A))$ is 1.  We will show that $A$ cannot be any $N^{i},$ $i \geq 2$, for any nilpotent Leibniz algebra $N$.  This is an extension of the Lie algebra result in \cite{chao}.  Suppose to the contrary that $A = N^{i}$, where $N$ is nilpotent of class $t$, and let $z$ be a basis for $Z_{1}(A)$.  
For $n \in N$, $nz = \alpha_{nz}z$ and $zn = \alpha_{zn}z$.  If one of these coefficients is not 0, then $N$ is not nilpotent.  Hence, $Z_{1}(A) \subset Z_{1}(N)$.  Since $N^{t}$ is an ideal in $A$, $N^{t} \subset Z_{1}(A)$.  Then, since $\dim(Z_{1}(A)) = 1,$ $N^{t} = Z_{1}(A)$.  Our initial assumptions guarantee that $N^{t-1} \subset A$.  Hence there exists a $y \in N^{t-1} \subset A$, $y \notin Z_{1}(A)$, such that $yu = \alpha_{yu}z$ and $uy = \alpha_{uy}z$ for all $u \in N$.  Let $w$ also be in $N$.  Then $y(uw) = (yu)w + u(yw) = \alpha_{yu}zw + u\alpha_{yw}z = 0$.  
Similarly $(uw)y = 0$.  Since $A \subset N^{i}$, it follows that $y$ is in the center of $A$.  Since $y$ and $z$ are linearly independent, this is a contradiction.  Hence, we have the following theorem.

\begin{theorem}
Let $A$ be a nilpotent nonabelian Leibniz algebra with one-dimensional center.  Then $A$ cannot be any $N^{i},$ $i \geq 2,$ for any nilpotent Leibniz algebra $N$.
\end{theorem}


Let $A$ be a Leibniz algebra.  define $R_{1}(A) = \{r \in A : Ar = 0\}$ and inductively, $R_{j+1}(A) = \{r \in A : Ar \subset R_{j}(A)\}$.  The $R_{j}(A)$ are left ideals of $A$.  
Let $B$ be an ideal in $A$ and let $L$ be the homomorphism from $A$ into the Lie algebra of derivations of $B$ given by $L(a) = L_{a}$, left multiplication of $B$ by $a$.  
Let $E(B,A)$ be the image of $L$.  $E(B,A)$ is a Lie algebra, and $E(B,B)$ is an ideal in $E(B,A)$.  For any left ideal, $C$, of $A$ that is contained in $B$, let $E(B,A,C) = \{E \in E(B,A) : E(C) = 0 \}$.  $E(B,A,C)$ is an ideal in $E(B,A)$, and $E(B,A)/E(B,A,C)$ is isomorphic to $E(C,A)$.

\begin{theorem}
Let $B$ be a nilpotent Leibniz algebra with $\dim(R_{1}(B)) = 1$ and $\dim(B) \geq 2$.  Then $B$ is not an ideal of any Leibniz algebra $A$ in which $B \subset \Phi(A)$.
\end{theorem}

\begin{proof}
Suppose that $B$ is an ideal in $A$ that contradicts the theorem.  Then $E(B,B) = L(B) \subset L(\Phi(A)) \subset \Phi(L(A)) = \Phi(E(B,A))$.  Let $\{z_{1}, z_{2}, \hdots, z_{k}\}$ be a basis for $R_{2}(B)$ and $\{z_{k}\}$ be a basis for $R_{1}(B)$.  Let $\Pi$ be the restriction map from $E(B,A)$ to $E(R_{2}(B),A)$.  Since $\left(E(B,B) + E(B,A,R_{2}(B))\right)/E(B,A,R_{2}(B)) = E(B,B)/\left(E(B,A,R_{2}(B)) \cap E(B,B)\right) = E(B,B)/E(B,B,R_{2}(B)) = E(R_{2}(B),B)$, it follows that $E(R_{2}(B),B) = \Pi \left(E(B,B)\right) \subset \Pi \left(\Phi(E(B,A)) \right) \subset \Phi \left( E(R_{2}(B),A) \right)$.

Now we show that $E(R_{2}(B),B)$ is not contained in $\Phi(E(R_{2}(B), A))$ by showing that $E(R_{2}(B), B)$ is complemented by a subalgebra in $E(R_{2}(B),A)$.  For $i = 1, \hdots, k$, let $e_{i}(z_{j}) = \delta_{ij}z_{k}$ for $j = i, \hdots, k$, where $\delta_{ij}$ is the Kronecker delta.  Let $S = \spn \{ e_{1}, \hdots, e_{k-1}\}$.  We claim that $S = E(R_{2}(B), B)$.  Since $BR_{2}(B) \subset R_{1}(B)$, it follows that $E(R_{2}(B),B) \subset S$.  To show equality, we show $\dim(E(R_{2}(B),B)) = k-1 = \dim(S)$.  For $x \in B$, $L_{x}$ induces a linear functional on $R_{2}(B)$; that is, $E(R_{2}(B),B)$ is contained in the dual of $R_{2}(B)$.  Hence $\dim(E(R_{2}(B),B) = \dim(R_{2}(B)) - \dim(R_{2}(B)^{B})$ where $R_{2}(B)^{B} = \{z \in R_2(B) : L_{x}(z) = 0 \text{ for all } x \in B \} = R_{1}(B)$.  Hence $\dim(E(R_{2}(B),B)) = k-1 = \dim(S)$, and $S = E(R_{2}(B),B)$.

We now show that $S$ is complemented in $E(R_{2}(B),A)$.  Let $M = \{E \in E(R_{2}(B),A) : E(z_{i}) = \sum_{j=1}^{k-1} \lambda_{ij}z_{j},$ $\lambda_{ij} \in F,$ $i = 1, \hdots, k-1\}$.  Clearly, $M$ is a subspace of $E(R_{2}(B),A)$ and $M \cap S = 0$.  We claim that $M+S = E(R_{2}(B),A)$.  Let $E \in E(R_{2}(B),A)$.  Then $E(z_{i}) = \sum_{j=1}^{k-1} \lambda_{ij}z_{j} + \lambda_{ik}z_{k}$ for $i = 1, \hdots, k-1$ and $E(z_{k}) = \lambda_{k}z_{k}$.  Now $E = (E - \sum_{i=1}^{k-1} \lambda_{i,k}e_{i}) + (\sum_{i=1}^{k-1} \lambda_{i,k}e_{i}) \in M +S$.  Therefore, $E(R_{2}(B),A) = M + S$.  Suppose that $M = 0$.  Then $E(R_{2}(B),A) = E(R_{2}(B),B)$, which contradicts $E(R_{2}(B),B) \subset \Phi \left(E(R_{2}(B),A)\right)$.  Thus, $M \neq 0$.  Hence, $S$ is complemented in $E(R_{2}(B),A)$, which contradicts $S \subset \Phi \left(E(R_{2}(B),A)\right)$.  This contradiction establishes the result.
\end{proof}

\noindent{\bf Acknowledgements.}  This work was completed as part of a Research Experience for Graduate Students at North Carolina State University, supported by the National Science Foundation.  The authors would like to thank Professor E. L. Stitzinger for his guidance and support.

\end{document}